\documentclass[11pt]{amsart}
\input{preamble}

\let\oldeqref\eqref
\makeatletter
\RenewDocumentCommand\eqref{s m}{%
  \IfBooleanTF#1%
  {\textup{\tagform@{\ref*{#2}}}}
  {\oldeqref{#2}}
}
\makeatother

\newcommand{\Card}{{\operatorname{Card}}}

\newcommand{\FE}[1]{{\mathcal {FE}({#1})}}

\newcommand{\Amp}{{\operatorname{Amp}}}

\newcommand{\BAmp}{{\operatorname{BAmp}}}

\newcommand{\Pos}{{\operatorname{Pos}}}

\newcommand{\mld}{{\operatorname{mld}}}

\newcommand{\Mth}{{\operatorname{Mon}^{2,\lt}_{\operatorname{Hdg}}}}

\newcommand{\C}{{\mathbb C}}

\newcommand{\Q}{{\mathbb Q}}
\newcommand{\R}{{\mathbb R}}

\newcommand{\Z}{{\mathbb Z}}

\newcommand{\coeff}{{\operatorname{coeff}}}

\newcommand{\Def}{{\operatorname{Def}}}

\newcommand{\Hdg}{\operatorname{Hdg}}

\newcommand{\lt}{{\rm{lt}}}

\newcommand{\Pic}{\operatorname{Pic}}

\newcommand{\reg}{{\operatorname{reg}}}

\renewcommand{\to}[1][]{\xrightarrow{\ #1\ }}

\makeatletter
\newcommand*{\da@rightarrow}{\mathchar"0\hexnumber@\symAMSa 4B }
\newcommand*{\da@leftarrow}{\mathchar"0\hexnumber@\symAMSa 4C }
\newcommand*{\xdashrightarrow}[2][]{%
  \mathrel{%
    \mathpalette{\da@xarrow{#1}{#2}{}\da@rightarrow{\,}{}}{}%
  }%
}
\newcommand{\xdashleftarrow}[2][]{%
  \mathrel{%
    \mathpalette{\da@xarrow{#1}{#2}\da@leftarrow{}{}{\,}}{}%
  }%
}
\newcommand*{\da@xarrow}[7]{%
  \sbox0{$\ifx#7\scriptstyle\scriptscriptstyle\else\scriptstyle\fi#5#1#6\m@th$}%
  \sbox2{$\ifx#7\scriptstyle\scriptscriptstyle\else\scriptstyle\fi#5#2#6\m@th$}%
  \sbox4{$#7\dabar@\m@th$}%
  \dimen@=\wd0 %
  \ifdim\wd2 >\dimen@
    \dimen@=\wd2 %
  \fi
  \count@=2 %
  \def\da@bars{\dabar@\dabar@}%
  \@whiledim\count@\wd4<\dimen@\do{%
    \advance\count@\@ne
    \expandafter\def\expandafter\da@bars\expandafter{%
      \da@bars
      \dabar@ 
    }%
  }%
  \mathrel{#3}%
  \mathrel{%
    \mathop{\da@bars}\limits
    \ifx\\#1\\%
    \else
      _{\copy0}%
    \fi
    \ifx\\#2\\%
    \else
      ^{\copy2}%
    \fi
  }%
  \mathrel{#4}%
}
\makeatother

\newtheoremstyle{citing}
  {}
  {}
  {\itshape}
  {}
  {\bfseries}
  {\textbf{.}}
  {.5em}
  {\thmnote{#3}}

\theoremstyle{plain}

\newtheorem{theorem}[subsection]{Theorem}

\newtheorem{lemma}[subsection]{Lemma}
\newtheorem{corollary}[subsection]{Corollary}

\newtheorem{proposition}[subsection]{Proposition}

\theoremstyle{remark}

\theoremstyle{definition}
\newtheorem{conjecture}[subsection]{Conjecture}
\newtheorem{definition}[subsection]{Definition}

\numberwithin{equation}{section}

\theoremstyle{remark}
\newtheorem{remark}[subsection]{Remark}

{\theoremstyle{citing}
}

\makeatletter
\newsavebox\myboxA
\newsavebox\myboxB
\newlength\mylenA

\newcommand*\xtilde[2][0.8]{%
    \sbox{\myboxA}{$\m@th#2$}%
    \setbox\myboxB\null
    \ht\myboxB=\ht\myboxA%
    \dp\myboxB=\dp\myboxA%
    \wd\myboxB=#1\wd\myboxA
    \sbox\myboxB{$\m@th\widetilde{\copy\myboxB}$}
    \setlength\mylenA{\the\wd\myboxA}
    \addtolength\mylenA{-\the\wd\myboxB}%
    \ifdim\wd\myboxB<\wd\myboxA%
       \rlap{\hskip 0.5\mylenA\usebox\myboxB}{\usebox\myboxA}%
    \else
        \hskip -0.5\mylenA\rlap{\usebox\myboxA}{\hskip 0.5\mylenA\usebox\myboxB}%
    \fi}

\newbox\usefulbox

\def\getslant #1{\strip@pt\fontdimen1 #1}

\def\xxtilde #1{\mathchoice
 {{\setbox\usefulbox=\hbox{$\m@th\displaystyle #1$}%
    \dimen@ \getslant\the\textfont\symletters \ht\usefulbox
    \divide\dimen@ \tw@ 
    \kern\dimen@ 
    \xtilde{\kern-\dimen@ \box\usefulbox\kern\dimen@ }\kern-\dimen@ }}
 {{\setbox\usefulbox=\hbox{$\m@th\textstyle #1$}%
    \dimen@ \getslant\the\textfont\symletters \ht\usefulbox
    \divide\dimen@ \tw@ 
    \kern\dimen@ 
    \xtilde{\kern-\dimen@ \box\usefulbox\kern\dimen@ }\kern-\dimen@ }}
 {{\setbox\usefulbox=\hbox{$\m@th\scriptstyle #1$}%
    \dimen@ \getslant\the\scriptfont\symletters \ht\usefulbox
    \divide\dimen@ \tw@ 
    \kern\dimen@ 
    \xtilde{\kern-\dimen@ \box\usefulbox\kern\dimen@ }\kern-\dimen@ }}
 {{\setbox\usefulbox=\hbox{$\m@th\scriptscriptstyle #1$}%
    \dimen@ \getslant\the\scriptscriptfont\symletters \ht\usefulbox
    \divide\dimen@ \tw@ 
    \kern\dimen@ 
    \xtilde{\kern-\dimen@ \box\usefulbox\kern\dimen@ }\kern-\dimen@ }}%
 {}}

\newcommand*\xoverline[2][0.75]{%
    \sbox{\myboxA}{$\m@th#2$}%
    \setbox\myboxB\null
    \ht\myboxB=\ht\myboxA%
    \dp\myboxB=\dp\myboxA%
    \wd\myboxB=#1\wd\myboxA
    \sbox\myboxB{$\m@th\overline{\copy\myboxB}$}
    \setlength\mylenA{\the\wd\myboxA}
    \addtolength\mylenA{-\the\wd\myboxB}%
    \ifdim\wd\myboxB<\wd\myboxA%
       \rlap{\hskip 0.5\mylenA\usebox\myboxB}{\usebox\myboxA}%
    \else
        \hskip -0.5\mylenA\rlap{\usebox\myboxA}{\hskip 0.5\mylenA\usebox\myboxB}%
    \fi}

\def\xxoverline #1{\mathchoice
 {{\setbox\usefulbox=\hbox{$\m@th\displaystyle #1$}%
    \dimen@ \getslant\the\textfont\symletters \ht\usefulbox
    \divide\dimen@ \tw@ 
    \kern\dimen@ 
    \overline{\kern-\dimen@ \box\usefulbox\kern\dimen@ }\kern-\dimen@ }}
 {{\setbox\usefulbox=\hbox{$\m@th\textstyle #1$}%
    \dimen@ \getslant\the\textfont\symletters \ht\usefulbox
    \divide\dimen@ \tw@ 
    \kern\dimen@ 
    \xoverline{\kern-\dimen@ \box\usefulbox\kern\dimen@ }\kern-\dimen@ }}
 {{\setbox\usefulbox=\hbox{$\m@th\scriptstyle #1$}%
    \dimen@ \getslant\the\scriptfont\symletters \ht\usefulbox
    \divide\dimen@ \tw@ 
    \kern\dimen@ 
    \xoverline{\kern-\dimen@ \box\usefulbox\kern\dimen@ }\kern-\dimen@ }}
 {{\setbox\usefulbox=\hbox{$\m@th\scriptscriptstyle #1$}%
    \dimen@ \getslant\the\scriptscriptfont\symletters \ht\usefulbox
    \divide\dimen@ \tw@ 
    \kern\dimen@ 
    \xoverline{\kern-\dimen@ \box\usefulbox\kern\dimen@ }\kern-\dimen@ }}%
 {}}
\makeatother

{\theoremstyle{definition}

\title{Footnotes to the birational geometry of primitive symplectic varieties}

\author{Christian Lehn}
\address{Christian Lehn\\ Fakult\"at f\"ur Mathematik\\ Technische Universit\"at Chemnitz\\
Reichenhainer Stra\ss e 39, 09126 Chemnitz, Germany}
\email{christian.lehn@mathematik.tu-chemnitz.de}

\author{Giovanni Mongardi}
\address{Giovanni Mongardi\\Alma Mater Studiorum, Università di Bologna,  P.zza di porta san Donato, 5, 40126 Bologna, Italia}
\email{giovanni.mongardi2@unibo.it}

\author{Gianluca Pacienza}
\address{Gianluca Pacienza\\ 
Universit\'e de Lorraine, CNRS, IECL\\
F-54000 Nancy -- France }
\email{gianluca.pacienza@univ-lorraine.fr}
\begin{document}

\begin{abstract}
    In this note, we extend to the singular case some results on the birational geometry of irreducible holomorphic symplectic manifolds.
\end{abstract}




\maketitle
\bigskip

\hfill{\it{In memoria di Alberto Collino}}
\setcounter{tocdepth}{1}
\tableofcontents

\section{Introduction}
Irreducible holomorphic symplectic (IHS) manifolds and their singular generalizations (the so-called primitive symplectic varieties) form a relatively well-studied class of varieties. Apart from their intrinsic interest, they offer the possibility to test general conjectures in algebraic or complex geometry as on the one hand, they are well-behaved, and on the other hand, they possess a sufficiently rich geometry. Their birational geometry is particularly interesting and provides an instance of this principle. As a sample of this (and taking the risk of leaving many important results aside), fibers of the period map are exactly birational IHS manifolds (cf. \cite{Ver13}, \cite[Corollary 6.1]{Huy12}), the MMP of moduli spaces of sheaves on a $K3$ surface has been studied systematically via wall-crossing with respect to Bridgeland stability conditions (cf. \cite{BM14mmp}), and   
termination of flips or the Morrison--Kawamata cone conjecture, which are open problems in general, are already known for IHS manifolds, cf. \cite{LP16}, \cite{AV17}. 

In this note, we discuss three aspects of the birational geometry of primitive symplectic varieties that have been the object of considerable recent interest: effective birationality, termination of flips, and monodromy invariance of wall divisors, which we present in detail below.

\subsection*{Effective Birationality}
First, we concentrate on effective birationality. The problem of boundedness for the birationality of pluricanonical maps has attracted much attention in the last 10 years (see the papers \cite{HaconACC,Bir20} and the references therein, for results in the framework of irreducible holomorphic symplectic varieties see \cite{Charles16}). Here, we are able to provide an effective version of such results in the most general setting, for primitive symplectic varieties. Recall that by the Bogomolov--Beauville--Fujiki form, the second cohomology (more precisely, the torsion free part) of such a variety embeds into its dual. We denote by $A_X:=H^2(X,\Z)^\vee/H^2(X,\Z)$ the \emph{discriminant group}.

\begin{theorem}\label{thm:A}
Let X be a projective primitive symplectic variety of dimension $2n$ and let $L\in \Pic(X)$ be a big line bundle on it. 
Then for all 
\begin{equation}\label{eq:eff} 
m\geq \frac{1}{2}(2n+2)(2n+3)[(4 \Card(A_X))^{\rho(X)-1}] !,
\end{equation}
the map associated to the linear system $|mL|$
is birational onto its image. 
\end{theorem}
If we replace $\rho(X)$ with $h^{1,1}(X)$ in equation (\ref{eq:eff}), we obtain an effective bound that holds for the whole family of deformations of $X$. Notice that in particular, for any integer $n$ and positive constant $C$, the family of all projective primitive symplectic varieties of dimension $2n$ and fixed deformation type, endowed with a big line bundle of volume at most $C$, is birationally bounded. Therefore, the theorem can be regarded an effective version of Birkar's birational boundedness theorem for pluri log-canonical maps (cf. \cite[Corollary 1.4]{Bir20}) in the setting of primitive symplectic varieties. In the smooth case, the result was obtained in \cite[Corollary 1.3]{KMPP19}. The key for the generalization to primitive symplectic varieties is the recent study of the numerical and geometric properties of prime exceptional in this setting, cf. \cite{LMP21, LMP22}.

Notice that for deformations of moduli spaces of sheaves, the bound is even more explicit.
\begin{corollary}\label{cor:moduli}
Let X be a projective primitive symplectic variety which is deformation equivalent to a moduli space of sheaves $M_v(S)$ with respect to some $v$-generic polarization $H$ on a projective $K3$ surface $S$ (respectively to $K_v(S)$ in case $S$ is abelian), for a Mukai vector  $v=aw$ on $S$ with $w^2=2k$ such that $(S,v,H)$ is an $(a,k)$--triple in the sense of \cite[Definition 1.3]{peregorapagnetta20}. 
Then for any big line bundle $L\in \Pic(X)$, the map associated to the linear system $|mL|$ is birational onto its image for any $$m\geq \frac{1}{2}(\dim(X)+2)(\dim(X)+3)[(8k)^{\rho(X)-1}] !.$$
\end{corollary}
Recall that with the notation of the above theorem $\dim M_v(S)= 2a^2k+2\epsilon$, where $\epsilon = 1$ (respectively $\epsilon=-1$) if $S$ is a $K3$ (resp. an abelian) surface (see \cite[Theorem 4.4]{KLS06}.
\subsection*{Termination of Flips}
Termination of flips is one major open problem in the MMP. A well-known general strategy due to Shokurov \cite{ShoV} involving the study of some regularity properties of invariants of the singularities (the so-called minimal log-discrepancies) has allowed to prove termination for IHS manifolds \cite{LP16}. Due to the generalization to hyperquotient singularities of the lower semicontinuity of minimal log-discrepancies obtained in \cite{NS20} and to the advances in the theory of primitive symplectic varieties, we remark here that Shokurov's strategy works now in a more general framework. 

\begin{theorem}\label{thm:termination}
Let X be a projective primitive symplectic variety with terminal hyperquotient singularities. 
Let $\Delta$ be an effective $\mathbb R$-divisor on $X$, such that the pair $(X,\Delta)$ is log-canonical.
Then every log-MMP for  $(X,\Delta)$ terminates in a minimal model
 $(X',\Delta')$ where $X'$ is a symplectic variety with (canonical) singularities and $\Delta'$
is an effective, nef $\mathbb R$-divisor.
\end{theorem}
Recall that a variety $Y$ has hyperquotient singularities if $Y$ is a locally complete intersection inside an ambient variety which is the quotient of a smooth variety by a finite group action (see \cite[Theorem 1.2]{NS20} for a slightly more general definition). 

In particular, we immediately deduce the following for an important class of singular IHS varieties which has received much attention in recent years cf. e.g. \cite{Men20, menetriess, FM21, menetriess2}.
\begin{corollary}
Let $X$ be an IHS orbifold (in the sense of \cite[Definition 3.1]{Men20}. Let $\Delta$ be an effective $\mathbb R$-divisor on $X$, such that the pair $(X,\Delta)$ is log-canonical.
Then every log-MMP for  $(X,\Delta)$ terminates.
\end{corollary}

\subsection*{Invariance of walls under parallel transport}
In the study of the birational geometry of IHS manifolds, wall divisors play a central role, analogous to the role played by $-2$ curves on K3 surfaces. As a final contribution, we extend an important property of wall divisors to the most general singular setting.  The following theorem is the content of Section~\ref{wall}:
\begin{theorem}\label{thm:walls}
Let $X$ be a projective primitive symplectic variety, and let $D\in Pic(X)$ be a wall divisor on it. Then $D$ remains a wall divisor under parallel transport, and there exists $\varphi\in \Mth(X)$ such that $\varphi(D)$ is dual to an extremal curve. Conversely, any divisor dual to an extremal curve is a wall divisor.
\end{theorem}
Notice that under the additional hypotheses of $\mathbb Q$-factoriality and terminality (and $b_2\geq 5$) the result above was proved in \cite[Proposition 7.4]{LMP22}, using, among other things, the fact that under these hypotheses reflections in prime exceptional divisors are integral. This is trivially false in the non-terminal case \cite[Example 3.12]{LMP22}. Hence, here we have to take a different path to tackle the general case. It is also important to notice that the definition of wall divisors in this larger framework (cf. Definition \ref{definition wall divisor}) includes an extra condition which can be deduced from the Morrison--Kawamata cone conjecture when $b_2 \geq 5$ and the variety is $\Q$-factorial and terminal, see Remark \ref{rmk:third condition} for a detailed discussion of this point.

\subsection*{Acknowledgments.} 
We thank the referee whose accurate observations pushed us to improve the content and the presentation of the paper, and Francesco Denisi for useful exchanges. 
C.L. was supported by the DFG research grant Le 3093/3-2 and by the SMWK research grant SAXAG. G.M. was supported by PRIN2020 research grant "2020KKWT53" and is a member of the INDAM-GNSAGA. G.P. was supported by the CNRS International Emerging Actions (IEA) project ``Birational and arithmetic aspects of orbifolds''.

\section{Preliminaries}
We start by recalling the main definitions about singular symplectic varieties. 

Let $X$ be an irreducible complex variety.
Recall that, for any integer $p\geq 1$, the sheaf of {\it reflexive holomorphic $p$-forms} on $X$
is $\Omega_X^{[p]}:=(\Omega_X^{p})^{**}$. If $X$ is normal, it can be alternatively (and equivalently) defined as $\iota_* \Omega_{X_{\reg}}^{p}$, where 
$\iota : X_{\reg}\hookrightarrow X$
is the inclusion of the regular locus of $X$.

Recall the following definition due to Beauville \cite{Bea00}. 
\begin{definition}\label{def:symp-sing}
Let X be a normal 
variety.
\begin{enumerate}
\item[i)] A \emph{symplectic form} on X is a closed reflexive $2$-form $\sigma$ on $X$
which is non-degenerate at each point of $X_{\reg}$.
\item[ii)]  If $\sigma$ is a symplectic form on $X$, the variety $X$ has \emph{symplectic
singularities} if for one (hence for every) resolution $f : \tilde X \to X$ of the singularities
of $X$, the holomorphic symplectic form $\sigma_{\reg} :=\sigma_{|X_{\reg}}$ extends to
a holomorphic $2$-form on $\tilde X$. In this case, the pair $(X,\sigma)$ is called \emph{symplectic variety}.
\item[iii)] A \emph{primitive symplectic variety} is a normal compact K\"ahler variety $X$ such that $h^1(X,\mathcal O_X)=0$ and $H^0(X,\Omega_X^{[2]})$ is generated by a holomorphic symplectic form $\sigma$ such that $X$ has symplectic singularities.
\end{enumerate}
\end{definition}

For the definition and basic properties of K\"ahler forms on possibly singular complex spaces, we refer the reader to,
for example, \cite[Section 2]{BL18}.
For a normal variety $X$ such that $X_{\reg}$ has a symplectic form $\sigma$, Beauville's condition above that the pullback of $\sigma$ to a resolution of $X$ extends as a regular $2$-form is in fact equivalent to having canonical, even rational singularities by \cite{Elk81}, \cite[Corollary~1.7]{KS21}. 
Let $X$ be a primitive symplectic variety. Then, there is a quadratic form $q_X$ on $H^2(X,\C)$, the so-called Beauville--Bogomolov--Fujiki form (\emph{BBF form} for short), see \cite[Definition~5.4]{BL18}. Up to scaling, it is defined by the formula
\begin{equation}\label{eq bbf}
q_X(\alpha):= \frac{n}{2} \int_X \left(\sigma\bar\sigma\right)^{n-1}\alpha^2 + (1-n) \int_X \sigma^n\bar\sigma^{n-1}\alpha \int_X \sigma^{n-1}\bar\sigma^n\alpha.
\end{equation}
Due to \cite{Nam01b, Kir15, Mat15, Sch20, BL18} we know that $q_X$ is defined over $\Z$ and nondegenerate of signature $(3,b_2(X)-3)$, see Section~5 of \cite{BL18} and references therein. It is used to formulate the local Torelli theorem \cite[Proposition~5.5]{BL18}, satisfies the Fujiki relations \cite[Proposition~5.15]{BL18}, and is compatible with the Hodge structure on $H^2(X,\Z)$.  Moreover, it allows to identify second degree homology with cohomology, more precisely:

\begin{definition}\label{definition dual class}
Let $X$ be a primitive symplectic variety. For $\alpha \in H^2(X,\Q)$, we define the \emph{dual class} $\alpha^\vee\in H^2(X,\Q)^\vee = H_2(X,\Q)$ by the condition
\[
q_X(\alpha,\beta) = \alpha^\vee(\beta).
\]
In the same way, we define $\gamma^\vee\in H^2(X,\Q)$ for a homology class $\gamma\in H_2(X,\Q)$. Clearly, $\alpha^{\vee\vee} = \alpha$.
\end{definition}

We end this section by recalling some relevant cones.
\begin{definition}
Let $X$ be a projective primitive symplectic variety. We define the following cones inside $\Pic(X)\otimes \R$:
\begin{itemize}
    \item The ample cone $\Amp(X)$ is the cone generated by ample classes.
    \item The positive cone $\Pos(X)$ is the connected component of the cone of BBF positive classes containing $\Amp(X).$
    \item The birational ample cone $\BAmp(X)$ is the union of all pullbacks $f^*\Amp(X')$, where $X'$ is a primitive symplectic variety locally trivially deformation equivalent to $X$, and $f\,:X\dashrightarrow X'$ is a birational map such that $f_*$ induces an isomorphism on $\Q$-Cartier divisors.
    \item The fundamental exceptional chamber $\FE{X}$ is the cone inside $\Pos(X)$ whose elements have positive intersection with {\it prime exceptional divisors}, i.e. with prime Cartier divisors which have negative square with respect to the BBF form.
\end{itemize}

\end{definition}
\begin{remark}\label{remark:cones}
Notice that we have $\Amp(X)\subset \BAmp(X) \subset \FE{X}\subset \Pos(X)$, where most inclusions are by definition, and $\BAmp(X) \subset \FE{X}$ follows because the birational maps we consider send prime exceptional divisors into prime exceptional ones. When $X$ has $\Q$-factorial and terminal singularities, by \cite[Proposition 5.8]{LMP22}, we have $\widebar{\BAmp(X)}=\widebar{\FE{X}}$.
\end{remark}

\section{Effective birationality}
To prove effective birationality, we want to bound denominators in Boucksom--Zariski decompositions of divisors, whose definition is recalled below. The quadratic form and the definition below can be given for $\Q$-Weil divisors, see \cite[Sections~2.3, 2.4, and 3]{KMPP19}. Since it will only be used in the $\Q$-factorial case, for simplicity we restrict ourselves to this case (to which one can always reduce in the projective case). 
\begin{definition}
  Let $D$ be a pseudo-effective $\Q$-divisor on a $\Q$-factorial primitive symplectic variety $X$. A Boucksom--Zariski decomposition of $D$ is a decomposition
  $$
 D= P(D) + N(D),
$$
where $P(D)$ and $N(D)$ are $\Q$-divisors, called respectively the positive and negative part of $D$, satisfying the following:
\begin{enumerate}
\item[1)] $P(D)$ is $q_{X}$-nef, i.e. $q_{X}(P(D),E)\geq 0$ for all effective  divisors;
\item[2)] if not empty, the divisor $N(D)$ is $q_{X}$-exceptional, i.e. it is effective and  the Gram matrix   $(q_{X}(N_i,N_j))_{i,j}$ of the irreducible components of the support of $N(D)$ 
is negative definite;
\item[3)] $q_{X}(P(D),N(D))=0$;
\end{enumerate}

\end{definition}
The existence and unicity of the Boucksom--Zariski decomposition
is proven in \cite{KMPP19} for effective divisors. Thanks to work of Denisi \cite{Den22} on the support of the negative part of the decomposition and to results contained in \cite{LMP22}, we can deduce the pseudo-effective case from the effective one.

\begin{theorem}
Let $X$ be a projective primitive symplectic variety with terminal and $\Q$-factorial singularities and let $D$ be a pseudo-effective divisor on $X$. Then, there is a unique Boucksom--Zariski decomposition $D=P+N$, and we have a surjection 
$$H^0(X,\mathcal O_X(kP))\twoheadrightarrow H^0(X,\mathcal O_X(kD)) $$ for all integers $k$ such that $kP$ and $kD$ are integral.
\end{theorem}
\begin{proof}
Let $D$ be a pseudo-effective divisor on $X$. If it is effective, there is nothing to prove by \cite{KMPP19}. Suppose now that it is not effective. Take a family $D_t$ of big divisors with $\lim_{t\to 0} D_t = D$. Recall that the big cone is the interior of the pseudo-effective cone. Therefore by \cite{KMPP19}, we can write $D_t=P_t+N_t$ where $P_t$ and $N_t$ are the positive and negative part of the Boucksom--Zariski decomposition of $D_t$. It was proven in \cite[Section 4]{Den22}  that, when $X$ is a smooth IHS, the big cone can be decomposed in chambers in each of which the prime components of the support of the negative part of the Boucksom--Zariski decomposition are constant (Denisi proves more precisely that these chambers are locally rational polyhedral subcones, but we do not need this part). The only thing that is used in this argument is the fact that the Beauville--Bogomolov--Fujiki form is an intersection product, i.e. for any two distinct prime divisors $F,\,F'$ we have $q_X(F,F')\geq 0$. Since the latter remains true for primitive symplectic varieties by \cite[Theorem 3.11]{KMPP19}, Denisi's result holds in particular for $\Q$-factorial and terminal primitive symplectic varieties. Thus, without loss of generality, we can assume that for $t$ close to zero the support of the negative parts $N_t$ is constant, i.e. $N_t= \sum_{i=1}^r a_{i,t}N_i$.
Thus
there exists a limit $N:= \sum_{i=1}^r a_{i}N_i$ of the $N_t$, with $a_{i}:= \lim_{t\to 0} a_{i,t} $. We set $P:=D-N$, which is therefore a limit of the positive parts $P_t$.
By \cite[Proposition 5.8]{LMP22}, the limit $P$ of movable divisors $P_t$ is in the closure of the movable cone of $X$, which coincides with the $q_X$-nef cone (here we use the terminality and $\Q$-factoriality of $X$). So we have proven the existence of a Boucksom--Zariski decomposition for pseudo-effective divisors in this setting. For the unicity, it suffices to follow the original argument by Zariski \cite[Theorem 7.7, beginning of the proof]{Zar62}, which again only uses the fact that we have an intersection product and the Gram matrix generated by the prime components of the negative part is negative definite. 
The proof of the last statement is exactly as in \cite[Proposition 3.8]{KMPP19}, where the unicity is crucially used. 
\end{proof}

\begin{proposition}\label{prop:bound-denom}
Let X be a projective primitive symplectic variety of Picard number $\rho(X)$ with $\Q$-factorial and terminal singularities. Then denominators of the coefficients of the negative and positive parts of the Boucksom--Zariski
decompositions of all pseudo-effective Cartier divisors are  bounded by $(4\Card(A_X))^{\rho(X)-1} !$.  
\end{proposition}
\begin{proof}
We have to check the hypothesis of \cite[Corollary 4.11]{KMPP19}, that is, the following:
\begin{enumerate}
\item Any prime exceptional divisor in $X$ is contractible on a birational model which is locally trivially deformation equivalent.
\item Given a contraction of a prime exceptional divisor $E$, the general contracted curve $\ell$ is either a smooth $\mathbb{P}^1$ or the union of two smooth $\mathbb{P}^1$'s intersecting transversally in one point.
\item Given $\ell$ and $E$ as in the previous item, the class of $\ell$ is equal to $\frac{-2E^\vee}{q(E)}$, where $E^\vee$ is the dual class introduced in Definition \ref{definition dual class}.
\item For any prime exceptional divisor $E$, either $E$ or $E/2$ is primitive.
\end{enumerate}
From these items, the conclusion will follow immediately. Items (1), (2) and (4) are the content of \cite[Theorem 1.2, item (1)]{LMP21}. Notice that these items hold also without the terminality assumption on $X$. Item (3) is the content of \cite[Remark 3.11]{LMP22}, which holds also without the $\Q$-factoriality hypothesis. 
 
\end{proof}
\begin{remark}
The content of \cite[Remark 3.11]{LMP22} is a direct consequence of \cite[Theorem 3.10]{LMP22}, where it is proven that reflections along prime exceptional divisors are monodromy operators and, in particular, integral. However, there are counterexamples to the integrality of these reflections without the terminality assumption, already for K3 surfaces (see \cite[Example 3.12]{LMP22}, so it is not clear whether the third item above holds without this assumption.
\end{remark}

By the work of Perego and Rapagnetta \cite{peregorapagnetta20}, we can give an explicit bound for moduli spaces of sheaves on K3 or abelian surfaces: 

\begin{proposition}\label{prop:PR}
Let $S$ be a K3 (resp. abelian) surface and let  $v=mw$ be a Mukai vector
on it, with $w^2=2k$ such that $(S,v,H)$ is an $(m,k)$--triple in the sense of \cite[Definition 1.3]{peregorapagnetta20}. Let us consider the moduli space $X:=M_v(S)$ with respect to some $v$-generic polarization $H$ on $S$ (respectively $X=K_v(S)$ in the abelian case). Then $\Card(A_X)=2k$. In particular, denominators of the Zariski decomposition are bounded by $(8k)^{\rho(X)-1} !$. 
\end{proposition}
\begin{proof}
By \cite[Theorem 1.6]{peregorapagnetta20}, there is an isometry $w^\perp = v^\perp\cong H^2(X,\Z)$. In particular, it follows that the second cohomology of $X$ is isometric (as a lattice) to the second cohomology of $S^{[k-1]}$. Therefore, the discriminant group is cyclic of order $2k$ and the claim follows.
\end{proof}

From Proposition \ref{prop:bound-denom}, along the same lines of \cite{KMPP19}, we generalize the result of effective birationality to the natural singular setting arising from the MMP:

\begin{proof}[Proof of Theorem \ref{thm:A}]
First, we restrict to the case of $\Q$-factorial and terminal singularities.
Let $L$ be a big line bundle on $X$. 
Consider the Boucksom--Zariski decomposition 
$$
a L = P + N,
$$
where $P\in \overline{\mathcal{FE}_X}$ is a Cartier divisor, $N$ a Cartier  $q_X$-exceptional divisor and 
$$ 
a  = (4\Card(A_X))^{\rho(X)-1} !
$$
is the integer given by Proposition \ref{prop:bound-denom} and clearing the denominators in the Boucksom--Zariski decompositions.
Since $L$ is big by hypothesis, the positive part $P$ is big. 
By Remark \ref{remark:cones}, we have that 
$\overline{\mathcal{FE}_X}$ coincides with the closure of the birational ample cone. 
Therefore, there exists a projective primitive symplectic variety $X'$ with $\Q$ factorial and terminal singularities and a birational map
$$
\phi :X\dashrightarrow X'
$$
such that 
$$
P':= f_* P
$$
is an integral, nef and big divisor on $X'$.  
Now we can apply Koll\'ar's extension \cite[Theorem 5.9]{Kol95} to nef and big divisors of the Angehrn--Siu  result \cite{AS95} to $P'$, 
to deduce that 
for all $m\geq (dimX +2)(dim X+3)/2$
the morphism associated to $|mP'|$ is injective on $X'$.
As a byproduct we then obtain that the linear system $|amL|$
separates two generic points on $X$.

Now we assume that  $X$ has arbitrary singularities and let $D$ be a big and nef divisor on it. As $X$ is projective, by \cite[Corollary 1.4.3]{BCHM10} there exists a terminal and $\Q$-factorial variety $Y$ and a map $\pi\,:\,Y \to X$ which is a partial resolution of singularities. We can apply the conclusion of the previous part of the proof to $Y$, so that the map associated to $m\pi^{*}(D)$ is birational onto its image, for all $m$ as in (\ref{eq:eff}). By definition, the map associated to $m\pi^{*}(D)$ factors through $\pi$, therefore our claim holds. 
\end{proof}
\begin{proof}[Proof of Corollary \ref{cor:moduli}]
The Corollary follows immediately from Theorem \ref{thm:A} and Proposition \ref{prop:PR}. 
\end{proof}

\section{Termination of flips}
In this section, we discuss termination of flips for primitive symplectic varieties with terminal hyperquotient singularities. We start by recalling standard definitions and two main conjectures  in the field. 

A {\it log pair} $(X,\Delta)$ consists of a normal variety $X$ and a $\mathbb R$-Weil divisor $\Delta\geq 0$ such that $K_X+\Delta$ is $\mathbb R$-Cartier. 
A {\it log resolution} of a log pair $(X,\Delta)$ is a projective birational morphism 
$\pi: \tilde X\to X$ such that $\tilde X$ is smooth and $\pi^* \Delta + Exc(\pi)$ has simple normal crossing support.
A birational morphism $f:\tilde X\to X$ between varieties for which $K_X$ and $K_{\tilde X}$ are well-defined is called {\it crepant} if $\pi^*K_X=K_{\tilde X}$.
A {\it crepant resolution} is a resolution of singularities which is also a crepant morphism. 

If $(X,\Delta)$ is a log pair and $\pi:\tilde X\to X$ is a log-resolution of $(X,\Delta)$, then we define the log discrepancy $a(E,X,\Delta)$ for a divisor $E$ over $X$ by the formula
\[
K_{\tilde X} + \tilde \Delta = \pi^* (K_{X} +  \Delta) + \sum_{E \subset \tilde X}  (a(E,X,\Delta)-1) E,
\]
where $\tilde \Delta$ is the strict transform of $\Delta$.

Let $c_X(E)\in X$ be the center of a divisor over $X$. This is a not necessarily closed point of $X$. The \emph{minimal log discrepancy} at $x\in X$ is
\[
\mld(x,X,\Delta):=\inf_{c_X(E)=x}a(E,X,\Delta)
\]
and the minimal log discrepancy along a subvariety $Z\subset X$ is
\[
\mld(Z,X,\Delta):=\inf_{x\in Z}\mld(x,X,\Delta).
\]

Notice that from the definition we have that 
\begin{equation}\label{eq:reverse}
Z\subset Z' \Rightarrow \mld(Z,X,\Delta) \geq \mld(Z',X,\Delta).
\end{equation}
We refer to \cite[\S 1]{Am99} for more details.


Ambro and Shokurov have made the two following conjectures about mlds in \cite{Am99,ShoV}. The importance of these conjectures is that if they are fulfilled, then log-flips terminate by the main theorem of \cite{ShoV}.
\begin{conjecture}\label{hypo acc}
{\bf{(ACC)}} Let $\Gamma\subset [0,1]$ be a DCC-set, {\it i.e.,}  all decreasing sequences in $\Gamma$ become eventually constant. For a fixed integer $k$ the set
$$
\Omega_k:= \left\{ \mld(Z,X,\Delta)\middle|
\begin{array}{l}
\dim X = k,\\ 
(X,\Delta) \textrm{ log pair}\\
Z\subset X \textrm{ closed subvariety}\\
\coeff(\Delta)\in \Gamma
\end{array}
\right\}
$$
is an ACC-set, that is, every increasing sequence $\alpha_1 \leq \alpha_2 \leq \ldots$ in $\Omega_k$ eventually becomes stationary.
\end{conjecture}
\begin{conjecture}\label{hypo lsc}
{\bf{(LSC)}} Let $X$ be a normal $\Q$-Gorenstein variety and let $\Delta$ be an $\R$-Weil divisor on $X$ such that $K_X+\Delta$ is $\R$-Cartier. Then for each $d$ the function $\mld_{(X,\Delta)}: X^{(d)} \to \R \cup \{-\infty\}$ is lower semi-continuous.
\end{conjecture}


\begin{lemma}\label{lemma deformation of hyperquotient}
Let $X$ be a projective variety with terminal hyperquotient singularities, and let $Y$ be a locally trivial deformation of $X$. Then $Y$ has hyperquotient singularities.
\end{lemma}
\begin{proof}
By hypothesis, for any point $p\in X$, the local ring at $p$ is a finite quotient of a complete intersection local ring $R$. That means that $R$ is the quotient of a regular local ring of dimension $\dim X + k$ by $k$ elements. The same holds for its analytic completion, as the $k$ functions still generate the local ideal of $\widehat{\mathcal{O}}_{X,p}$, and their number cannot drop by dimensional reasons. Hence, the property of having hyperquotient singularities is invariant under analytically locally trivial deformations.
\end{proof}

\begin{proof}[Proof of Theorem \ref{thm:termination}]
We argue as in the proof of \cite[Theorem 1.2]{LP16} and rely on Shokurov's strategy \cite[Theorem]{ShoV} to prove termination. We refer the reader to \cite{LP16} for all the relevant definitions and further details. 

Let $Z$ be a variety appearing along a log-MMP for  $(X,\Delta)$. Let $Y$ be any $\Q$-factorial terminalization of $Z$. Since $Y$ is birational to $X$, $Y$ is locally trivially deformation equivalent to $X$ by \cite[Theorem 6.16]{BL21}. Thus, $Y$ has hyperquotient singularities by Lemma \ref{lemma deformation of hyperquotient}. Therefore, by \cite[Theorem 6.2]{NS20} the lower semi-continuity of the minimal log-discrepancy function holds on $Y$. 
By \cite[Theorem~3.8]{LP16}, the mld function on $Z$ is lower semi-continuous. 

By a theorem of Kawakita \cite{Kaw12}, the set of
all mlds for a fixed finite set of coefficients on a fixed projective variety is
finite. As in the course of the MMP, only a finite number of divisorial contractions can occur, we only have to care about the ACC condition along a sequence of flips.  By \cite[Theorem 6.16]{BL21}, any two $\Q$-factorial primitive  symplectic varieties birational in codimension one are 
locally trivially deformation equivalent. Therefore, the set of ACC condition is also verified. 

We conclude the termination of log-flips by  the main
theorem of \cite{ShoV}.
\end{proof}

\section{On the geometry of wall divisors}\label{wall}

In this section, we look at the shape of the birational ample cone for primitive symplectic varieties. To this end, we introduce a generalization of the definition of wall divisor contained in \cite{LMP22}, and we prove that they are related to extremal contractions.

\begin{definition}\label{definition wall divisor}
Let $X$ be a projective primitive symplectic variety and let $D\in \Pic(X)$. The divisor $D$ is called a wall divisor if the following three conditions hold:
\begin{enumerate}
    \item Negativity: $q_X(D)<0$
    \item Ample-non-orthogonality: for all $\varphi\in \Mth(X)$ we have $\varphi(D)^\perp \cap \BAmp(X)=\emptyset$.
    \item Semiample-orthogonality: there exists $\varphi\in \Mth(X)$ such that $\varphi(D)^\perp \cap \widebar{\BAmp}(X)\neq 0$, where the closure is taken in $\Pos(X)$.
\end{enumerate}
\end{definition}

\begin{remark}\label{rmk:third condition}
When $X$ is projective with $b_2\geq 5$ and has terminal and $\Q$-factorial singularities, the third condition is a trivial consequence of the proof of the Morrison--Kawamata birational cone conjecture (see \cite[Theorem~5.12]{LMP22}), as $D^\perp\cap \Pos(X)\neq 0$ due to the signature of the quadratic form on $\Pic(X)$, so by the transitive action of $\Mth(X)$ on the exceptional chambers of $\Pos(X)$,  there exists an element $\varphi\in \Mth(X)$ such that $\varphi(D)^\perp$ intersects the closure of the fundamental exceptional chamber, which then coincides with $\widebar{\BAmp}(X)$, by \cite[Proposition 5.8]{LMP22}.
This third condition will allow us to prove that wall divisors are parallel transport invariant using a strategy similar to \cite{BHT}. Without it, another way of proving  their parallel transport invariance could be via the use of twistor lines, which are not known to exist in general. This approach successfully used for instance in the orbifold case by Menet and Rie\ss$\,$  \cite{menetriess}.
\end{remark}

\begin{proposition}\label{prop:contracted are wall}
Let $X$ be a primitive symplectic variety and let $\pi:\, X\to Z $ be a relative Picard rank one contraction. Let $C$ be a contracted curve. Then the dual class $D$ to $C$ deforms to a wall divisor on all projective nearby deformations of $(X,D)$.
\end{proposition}
\begin{proof}
Without loss of generality, up to moving inside the Hodge locus of $C$, we can directly assume that $X$ is projective, as projective deformations are dense in any family which is non-trivial in moduli by \cite[Corollary 6.10]{BL18}. As the curve $C$ is contracted by $\pi$, we authomatically have that $C$ is orthogonal to $\pi^{*}(\Amp(Z))\subset\widebar{\BAmp(X)}$, thus the third condition of Definition \ref{definition wall divisor} is satisfied.
Let $\varphi\in\Mth(X)$. Since $\varphi$ is a Hodge isometry,  $[\varphi(C)]$ cannot be orthogonal to $\BAmp(X)$. To see this, suppose by contradiction that $\varphi(C)^\perp\cap \Amp(X')\not= \emptyset$, for some birational model $f:X\dashrightarrow X'$. Since $f_*$ is a Hodge isometry, by \cite[Lemma 5.17, item (2)]{Mar13}, either $f_*[\varphi(C)]$ or $-f_*[\varphi(C)]$ is a deformation of $[C]$ which remains of type $(2n-1,2n-1)$, i.e. it lies in the Hodge locus of $[C]$. As by \cite[Corollary 2.14]{LMP22}, the curve $C$ deforms in its Hodge locus, a multiple of $f_*[\varphi(C)]$ is effective and therefore cannot be orthogonal to $\Amp(X')$.
As $D$ is dual to $C$, the same holds for $\varphi(D)$ and the claim is proven.
\end{proof}

\begin{lemma}\label{lem:wall_are_extremal}
Let $X$ be a projective primitive symplectic variety and let $D$ be a wall divisor. Then, there exists $\varphi\in\Mth(X)$ and a birational model $X'$ such that $\varphi(D)$ is dual to an extremal curve on $X'$.
\end{lemma}
\begin{proof}
By the definition of wall divisor, it follows that $\varphi(D)^\perp$ does not intersect $\BAmp(X)$ but intersects $\widebar{\BAmp}(X)\cap \Pos(X)$ for some $\varphi\in\Mth(X)$. Therefore, there exists a birational model $X'$ of $X$ such that $\varphi(D)$ is orthogonal to a big and nef divisor $H'$. Without loss of generality, we can take $H'$ general in $\varphi(D)^\perp\cap\BAmp(X)$, so that $H'$ will lie in one facet of the ample cone of a birational model of $X$. By applying the base-point-free theorem to $(X', H')$, we contract only a curve (and its multiples), which is dual to $\varphi(D)$. 
\end{proof}
\begin{lemma}\label{lem:wall_are_deformable}
Let $X$ be a projective primitive symplectic variety and let $D$ be a wall divisor. Then, there exists an effective curve $C$ such that $\Hdg(C)=\Hdg(D)$ and $C$ deforms along its Hodge locus.
\end{lemma}

\begin{proof}
This is a direct consequence of Lemma \ref{lem:wall_are_extremal}: indeed, there exists an element $\varphi\in\Mth(X)$ and a birational model $X'$ such that $\varphi(D)$ is dual to an extremal curve $\varphi(C)$. By \cite[Corollary 2.14]{LMP22}, $\varphi(C)$ deforms in its Hodge locus (in $\Def(X')$ ), which contains also the pair $(X,C)$ (up to sign). Therefore $D$ is dual to the effective curve $\pm C$. 
\end{proof}
\begin{proposition}\label{prop:walls_deform}
Let $X$ be a projective primitive symplectic variety, and let $D$ be a wall divisor on $X$. Let $Y$ be another projective primitive symplectic variety such that there exists a locally trivial parallel transport operator $\varphi$ from $X$ to $Y$ such that $\varphi(D)\in \Pic(Y)$. Then $\varphi(D)$ is a wall divisor.
\end{proposition}
\begin{proof}
Up to changing birational model and acting with $\Mth(X)$, by Lemma \ref{lem:wall_are_deformable} we can assume, without loss of generality, that $D$ is dual to a curve $C$ (up to sign) which deforms along its Hodge locus and is contractible by a map $\pi\,:X\to \widebar{X}$. Therefore, $\varphi(D)$ is dual to (up to sign) an effective curve on $Y$. Let $\psi$ be any element in $\Mth(Y)$: we have that $\psi(\varphi(C))$ is still in the Hodge locus of $C$, hence also $\psi(\varphi(D))$ is dual to (up to sign) an effective curve. It remains to prove that the third condition of Definition \ref{definition wall divisor} holds: for this, it suffices to take a locally trivial deformation of the pair $(X,\pi)$, which amounts to a deformation of the contracted manifold $\widebar{X}$, along the same family which transports $D$ to $\varphi(D)$: this gives a deformation $\varphi(C)$ of the contracted curve $C$, which is dual to $\varphi(D)$ and is extremal on a (possibly) different birational model of $Y$ so that our claim follows. 
\end{proof}
Summing up Propositions \ref{prop:contracted are wall} and \ref{prop:walls_deform} and Lemma \ref{lem:wall_are_extremal}, we have proven Theorem \ref{thm:walls}.\\ 

The simplest example of wall divisors, as in the smooth case, is given by prime exceptional divisors. In the singular setting, it is not obvious that prime exceptional divisors are automatically wall divisors, as it is not proven that a prime exceptional divisor is contractible on a primitive symplectic variety which is locally deformation equivalent to the initial one. However, the geometry of prime exceptional divisors is particularly interesting, as by the following proposition they are always uniruled:

\begin{corollary}\label{lemma curve nonnegative square}
Let $X$ be a projective primitive symplectic variety, and let $E \subset X$ be a prime exceptional divisor. Then $E$ is uniruled and if  $R$ denotes a curve in the ruling of $E$, then $q(R)< 0$ and $[E] = \lambda [R]^\vee$ for some $\lambda >0$.
\end{corollary}
\begin{proof}
If $X$ is $\Q$-factorial, this is  \cite[Theorem 1.2, item (1)]{LMP21}. Otherwise, we may take a $\Q$-factorialization, i.e. a proper birational $\pi:Y\to X$ such that $Y$ is $\Q$-factorial and $\pi$ is small (and hence crepant), which exists by \cite[Corollary~1.4.3]{BCHM10}. Then $\pi^*E=\pi^{-1}(E)$ is an irreducible divisor of the same square and we apply \cite[Theorem~1.3]{LMP21} to $\pi^*E$ on $Y$. This yields that (a multiple of) $\pi^*E$ and its ruling curve $R'$ deform over the entire Hodge locus of $\pi^*E$ (or equivalently of $R'$) inside the space $\Def^\lt(Y)$ of locally trivial deformations of $Y$. By \cite[Proposition 2.12]{LMP22}, the space $\Def^\lt(X)$ sits canonically inside $\Def^\lt(Y)$ and the Hodge locus of $E$ in $\Def^\lt(X)$ can be identified with the common Hodge locus of $\pi^*E$ and the orthogonal complement to $H^2(X,\Z)$ inside of $H^2(Y,\Z)$. In particular, we may deform $(Y,\pi^*E)$ (up to taking a multiple) to a very general point $t$ of the Hodge locus of $E$ on $X$. The corresponding primitive symplectic variety $X_t$ will have Picard group of rank one generated by $E$. As $R'$ deforms along to a curve ruling a divisor $E'_t$ on $Y_t$, also $R$ deforms to a curve $R_t$ in $X_t$ (which is the image of some deformation of $R'$). Hence, the dual of $R_t$ must be a multiple of $E_t$ and this statement is invariant under deformation. Positivity of $\lambda$ can be deduced by pairing with a Kähler class.
\end{proof}

\bibliography{literature}
\bibliographystyle{alpha}

\end{document}